\documentclass[12pt,a4paper]{amsart}
\usepackage[all]{xy}
\usepackage{amsthm}
\usepackage{amssymb}
\usepackage{enumerate}
\usepackage{verbatim}
\usepackage{mathrsfs}
\usepackage{tikz-cd} 
\usepackage{fullpage}

\calclayout 

\numberwithin{equation}{section}

\theoremstyle{plain}

\newtheorem{lemma}{Lemma}[section]
\newtheorem{theorem}[lemma]{Theorem}
\newtheorem{proposition}[lemma]{Proposition}
\newtheorem{corollary}[lemma]{Corollary}
\newtheorem{claim}[lemma]{Claim}

\theoremstyle{definition}

\newtheorem{definition}[lemma]{Definition}

\newtheorem*{theorem*}{Theorem}
\newtheorem*{corollary*}{Corollary}

\theoremstyle{remark} 
\newtheorem{remark}[lemma]{Remark} 
\newtheorem*{claim*}{Claim} 

\renewcommand{\dim}{\operatorname{dim}}

\newcommand{\ep}{\varepsilon}
\newcommand{\Ext}{\operatorname{Ext}}

\renewcommand{\ge}{\geqslant}
\renewcommand{\geq}{\geqslant}

\newcommand{\half}{{\textstyle\frac{1}{2}}}
\newcommand{\shalf}{{\scriptstyle\frac{1}{2}}}

\newcommand{\Hom}{\operatorname{Hom}}

\renewcommand{\le}{\leqslant}
\renewcommand{\leq}{\leqslant}

\newcommand{\Rad}{\operatorname{\mathsf{Rad}}}

\newcommand{\Soc}{\mathsf{Soc}}
\newcommand{\Spec}{\operatorname{Spec}}

\newcommand{\StMod}{\operatorname{\mathsf{StMod}}}

\renewcommand{\setminus}{\smallsetminus}

\def\mcP{\mathcal{P}}

\def\bbA{\mathbb A}
\def\bbF{\mathbb F} 
\def\bbG{\mathbb G}

\def\bbZ{\mathbb Z} 

\def\mfm{\mathfrak m}

\newcommand{\bss}{\boldsymbol{s}}

\newcommand{\cP}{\mathscr{P}}

\title[A family of finite supergroup schemes]{Representations and cohomology of a family of finite
  supergroup schemes}

\author[Benson and Pevtsova]{Dave Benson and Julia Pevtsova}

\address{Dave Benson \\ 
Institute of Mathematics\\ 
University of Aber\-deen\\ 
King's College\\ 
Aber\-deen AB24 3UE\\ 
Scotland U.K.}

\address{Julia Pevtsova\\ 
Department of Mathematics\\ 
University of Washington\\ 
Seattle, WA 98195\\ 
U.S.A.}

\begin{document}

\begin{abstract} 
We examine the cohomology and representation theory of a family of finite supergroup
schemes of the form 
$(\bbG_a^-\times \bbG_a^-)\rtimes (\bbG_{a(r)}\times (\bbZ/p)^s)$. 
In particular, we show that a certain relation holds in the cohomology
ring, and deduce that for finite supergroup schemes having this as a
quotient, both cohomology mod nilpotents and projectivity
of modules is detected on proper sub-super\-group schemes. This 
special case feeds into the proof of a more general detection theorem for unipotent finite
supergroup schemes, in a separate work of the authors joint with
Iyengar and Krause.

We also completely determine the cohomology ring in the smallest
cases, namely $(\bbG_a^- \times \bbG_a^-) \rtimes \bbG_{a(1)}$ and
$(\bbG_a^- \times \bbG_a^-) \rtimes \bbZ/p$. The computation uses the
local cohomology spectral sequence for group cohomology, which we
describe in the context of finite supergroup schemes.
\end{abstract}


\date{\today}

\thanks{This work was supported by the NSF grant DMS-1440140 while the
  authors were in residence at the MSRI. 
The second author was 
partially supported by the DMS-0500946 award and by the Simons foundation}

\maketitle

\setcounter{tocdepth}{1}

\section{Introduction}

The calculations in this paper are motivated by the problem of detecting nilpotents 
in cohomology theories which has a long history. In algebraic topology, the celebrated nilpotence theorem 
in the stable homotopy category is due to Devinatz--Hopkins--Smith. 
For mod-$p$ finite group cohomology, Quillen showed 
that nilpotence is detected upon restriction to elementary abelian subgroups. Suslin proved an analogue of 
Quillen's detection theorem for cohomology of finite group schemes where the detection family consisted of 
abelian finite groups schemes isomorphic to $\bbG_{a(r)} \times (\bbZ/p)^s$ (preceded by the work of Suslin-Friedlander-Bendel \cite{Bendel/Friedlander/Suslin:1997b} on infinitesimal finite group schemes and  Bendel \cite{Bendel:2001a}
on unipotent finite group schemes). 

In joint work with Iyengar and Krause 
\cite{Benson/Iyengar/Krause/Pevtsova:bikp5}, we study the question of 
detecting nilpotents in the cohomology of a finite supergroup scheme, or
equivalently, a finite dimensional $\bbZ/2$-graded cocommutative Hopf superalgebra.
We establish a detecting family in the case of a unipotent finite supergroup scheme
which turns out to have a surprisingly more complicated structure than what one sees in 
the ungraded case in the detection theorems of Quillen and Suslin. A particularly difficult 
case arising in the course of the proof of the detection theorem in 
\cite{Benson/Iyengar/Krause/Pevtsova:bikp5}
is that of the degree two cohomology class determined 
by the central extension of $\bbG_a^- \times \bbG_{a(r)}\times
(\bbZ/p)^s$ by $\bbG_a^-$, where $\bbG_a^-$ is a supergroup scheme corresponding 
to the exterior algebra of a one dimensional super vector space concentrated in odd degree.  
The outcome of this paper, which feeds into the proof of the general result in 
\cite{Benson/Iyengar/Krause/Pevtsova:bikp5},
is that a certain product vanishes in cohomology but this relation does not 
follow in the usual way from the action of the Steenrod operations. 

In the course of producing the desired relation, we study the representation theory and cohomology of
finite supergroup schemes of the form $(\bbG_a^-\times \bbG_a^-)\rtimes
(\bbG_{a(r)} \times (\bbZ/p)^s)$, where the complement $\bbG_{a(r)} \times
(\bbZ/p)^s$ is acting faithfully on the normal sub-super\-group scheme $\bbG_a^- \times \bbG_a^-$.
We also obtain a great deal of information about the smallest case, computing 
the cohomology ring of  $(\bbG_a^-\times\bbG_a^-)\rtimes \bbZ/p$, which is our first result, proved in 
Section \ref{se:Ga(1)}. Note that for supergroup schemes, the cohomology is doubly graded: 
we write $H^{i,j}(G,k)$, where
the index $i \in \bbZ$ is cohomological, and the index $j\in\bbZ/2$
comes from the internal grading. 

\begin{theorem}[Theorem~\ref{th:Ga1} and Remark~\ref{re:Zp}]\label{th:Coho-ring}

Let $G$ be either $(\bbG_a^-\times\bbG_a^-)\rtimes \bbZ/p$ or
$(\bbG_a^-\times\bbG_a^-)\rtimes \bbG_{a(1)}$, each one being a semidirect product
with non-trivial action. Then the cohomology ring $H^{*,*}(G,k)$ is Gorenstein with the 
Poincar\'e series given by
\[ \sum_n t^n \dim_k H^{n,*}(G,k) = 1/(1-t)^2. \]
The algebra structure is given as follows. The generators are
\[ \zeta\in H^{1,1}(G,k),\ x\in H^{2,0}(G,k), \ \kappa\in
  H^{p,1}(G,k),\ \lambda_i\in H^{i,1+i}(G,k)\ (1\le i \le p-1). \] 
The relations are
\[ \lambda_i\zeta=0\ (1\le i\le p-1),\quad x\zeta^{p-1}=0, \quad
\lambda_i\lambda_j=\begin{cases} x\zeta^{p-2} & i+j=p \\
0 & \text{otherwise.} \end{cases} \]
\end{theorem}

One of the techniques we employ for this calculation is the local cohomology 
spectral  sequence which Greenlees~\cite{Greenlees:1995a} developed in the context of
cohomology of finite groups. This turns out to be essential to determine that  the product 
$\lambda_i\lambda_{p-i}$ is non-zero. For finite supergroup schemes, this spectral sequence takes
the form below, incorporating the modular function $\delta_G$ (see Section~\ref{se:local}). 

\begin{theorem}[Corollary~\ref{cor:local-coh} and Corollary~\ref{co:Gorenstein}]
Let $G$ be a finite supergroup scheme. Then there is a local
cohomology spectral sequence
\[ E_2^{s,t,j} = H^{s,t,j}_\mfm H^{*,*}(G,M) \Rightarrow
  H_{-s-t,j+\ep_G}(G,M \otimes\delta_G). \]
Here, the third index $j\in\bbZ/2$ is given by the internal grading, 
and $\delta_G$ is the \emph{modular function} of internal degree $\ep_G \in \mathbb Z/2$. 

Suppose that $\delta_G$ is trivial, which happens for example in the case 
where $G$ is unipotent. In this case, if $H^{*,*}(G,k)$ is Cohen--Macaulay
then it is Gorenstein, with shift $(0,\ep_G)$.
\end{theorem}

We also, along the way, make some computations of the structure of the
symmetric powers of a faithful two dimensional representation $V$ of
$\bbG_{a(r)} \times (\bbZ/p)^s$. We state it in terms of the dual 
$V^\#$, because we are interested in cohomology. In the case of
$(\bbZ/p)^s$ this is well known by restricting from
$\mathrm{SL}(2,p^s)$, whereas in the case of the Frobenius kernel, the 
results follow by restricting from $\mathrm{SL}_{2(r)}$ (see, for example, \cite[II.2.16]{Jantzen:2003a}). 
The following is a tabulation of the results proved in Section \ref{se:sympowers}.

\begin{theorem}\label{th:sympowers}
Let $V$ be a faithful two dimensional representation of $H=\bbG_{a(r)}
\times (\bbZ/p)^s$, $V^\#$ be the dual vector space, and $S^n(V^\#)$ be the module of degree $n$
polynomial functions on $V$,
\begin{itemize}
\item[\rm (i)] {\rm Periodicity:} For $n\ge p^{r+s}$ we have $S^n(V^\#)\cong kH
  \oplus S^{n-p^{r+s}}(V^\#)$, where $kH$ is the group algebra of the finite group scheme $H$.
\item[\rm (ii)] {\rm Projectivity:} $S^n(V^\#)$ is a projective module
  if and only if $n$ is congruent to $-1$ modulo $p^{r+s}$.
\item[\rm (iii)] {\rm Uniserial:} For $1\le i \le p-1$, the module
  $S^i(V^\#)$ is a uniserial module of dimension $i+1$. 
\item[\rm (iv)] {\rm Steinberg tensor product:} For $1\le i\le r+s$
  the module $S^{p^i-1}(V^\#)$ is isomorphic to the tensor product of
  Frobenius twists $S^{p-1}(V^\#) \otimes S^{p-1}(V^\#)^{(1)} \otimes
  \dots \otimes S^{p-1}(V^\#)^{(i-1)}$.
\item[\rm (v)] {\rm Rank variety:} The rank variety of
  $S^{p^i-1}(V^\#)$ is an explicitly described linear subspace of
  affine space $\bbA^{r+s}$ of codimension $i$.
\end{itemize}
\end{theorem}

Using Theorem \ref{th:sympowers} to make some spectral sequence computations, 
the following theorem is proved in Section \ref{se:proof}.

\begin{theorem}[Theorem~\ref{th:main}]
Let $k$ be a field of odd prime characteristic, and let $G$ be the
finite supergroup scheme $(\bbG_a^-\times \bbG_a^-)\rtimes(\bbG_{a(r)}
\times  (\bbZ/p)^s)$. Then there is a non-zero element $\zeta\in
H^{1,1}(G,k)$ such that
for all $u\in H^{1,0}(G,k)$ we have $\beta\cP^0(u).\zeta^{p^{r+s-1}(p-1)}=0$.
\end{theorem}

The following consequence will be used in our joint work with Iyengar
and Krause \cite{Benson/Iyengar/Krause/Pevtsova:bikp5}.

\begin{corollary}[Corollary~\ref{co:bikp5}]
Let $G$ be a finite unipotent supergroup scheme, with a normal sub-super\-group
scheme $N$ such that $G/N \cong \bbG_a^- \times \bbG_{a(r)} \times
(\bbZ/p)^s$. If the inflation map $H^{1,*}(G/N,k)\to H^{1,*}(G,k)$ is
an isomorphism and $H^{2,1}(G/N,k) \to H^{2,1}(G,k)$ is
not injective then there exists a non-zero element $\zeta\in
H^{1,1}(G,k)$ such that for all $u\in H^{1,0}(G,k)$ we have $\beta\cP^0(u).\zeta^{p^{r+s-1}(p-1)}=0$.
\end{corollary}

Throughout this paper, $k$ is a field of odd characteristic. 
Background on finite supergroup schemes can be found in the 
``sister paper" \cite{Benson/Iyengar/Krause/Pevtsova:bikp5}.
 We use \cite{Jantzen:2003a} as our standard reference for affine group schemes and 
their representations.  

\subsection*{Acknowledgements} 
This work started during Dave Benson's visit to the University of Washington in the Summer of 2016 as a distinguished visitor of the Collaborative Research Group in Geometric and Cohomological Methods  in Algebra of the Pacific Institute for Mathematical Sciences. We have enjoyed the hospitality of City University while working on this project in the summers of 2017 and 2018. Finally, we gratefully acknowledge the support and hospitality of the Mathematical Sciences Research Institute in Berkeley, California where we were in residence during the semester on ``Group Representation Theory and Applications" in the Spring of 2018.

\section{Semidirect products}\label{se:semidirect}

We begin by recalling, for example from Theorem 2.13 of Molnar \cite{Molnar:1977a},
the Hopf structure on the smash product of cocommutative Hopf algebras. The same
conventions work just as well in the graded cocommutative case, as follows.

Let $B$ be a graded cocommutative Hopf algebras, and $A$ be a Hopf
algebra which is an $B$-module bialgebra, then the tensor product 
coalgebra structure on the smash product $A \# B$ makes it a Hopf
algebra. In more detail, let $\tau\colon B \otimes A \to A$ be the map
giving the action. Then the multiplication on $A\# B$ is
\[ (a \otimes h)(b\otimes g) = \sum(-1)^{|h_{(2)}||b|}
  a\tau(h_{(1)},b)\otimes h_{(2)}g, \]
the comultiplication is
\[ \Delta(a \otimes h) = \sum (-1)^{|h_{(1)}||a_{(2)}|} (a_{(1)} \otimes h_{(1)}) \otimes
  (a_{(2)} \otimes h_{(2)}) \]
and the antipode is
\[ \bss(a \otimes h) = \sum (-1)^{(|a|+|h_{(1)}|)|h_{(2)}|}
  \tau(\bss(h_{(2)}),\bss(a)) \otimes \bss(h_{(1)}). \]
If $A$ is also graded cocommutative, we shall write $A \rtimes B$ for
this construction, and call it the \emph{semidirect product} of $A$
and $B$ with action $\tau$. There are obvious maps of Hopf algebras
\[ \xymatrix{A \ar[r] & A \rtimes B \ar@<.5ex>[r] & B \ar@<.5ex>[l]} \]
forming a split exact sequence.
Theorem~4.1 of the same paper implies that 
any split exact sequence of graded cocommutative Hopf algebras is
isomorphic to a semidirect product.

Recall that if $G$ is a finite supergroup scheme, then its group algebra $kG$ is 
defined as a linear dual to the coordinate algebra $k[G]$. Hence, it is a finite 
dimensional graded cocommutative Hopf algebra (see, for example, 
\cite{Benson/Iyengar/Krause/Pevtsova:bikp5} for more extensive background). 
We denote by $\bbG_a^-$ the supergroup scheme with the (self-dual) coordinate algebra $k[v]/(v^2)$ 
with $v$ an odd primitive element. Recall that $\bbG_{a(r)}$ is the $r$th Frobenius kernel of the 
additive group $\bbG_{a}$, a finite connected group scheme with coordinate algebra $k[T]/(T^{p^r})$ 
with $T$ primitive, and group algebra $k\bbG_{a(r)} = k[s_1,\dots,s_r]/(s_1^p,\dots,s_r^p)$.
The coproduct in $k[s_1,\dots,s_r]/(s_1^p,\dots,s_r^p)$
is given by  \[ \Delta(s_i)=S_{i-1}(s_1\otimes 1,\dots,s_i\otimes 1,\ 1\otimes
s_1,\dots,1\otimes s_i) \] 
where $S_0,S_1,\dots$ are the polynomials defining the addition of Witt vectors (see, for example, \cite[Appendix]{Benson/Iyengar/Krause/Pevtsova:bikp5} for more details). In the 
context of supergroup schemes, we think of $k\bbG_{a(r)}$ as concentrated in even degree.

Getting back to the discussion of the semidirect product, we are interested in the specific case 
where $A$ is the group algebra of $\bbG_a^- \times \bbG_a^-$,
the exterior algebra on two primitive generators $u$ and $v$,
and $B=kH$ is the group algebra of the finite group scheme  $H=\bbG_{a(r)} \times (\bbZ/p)^s$.
Here, either $r$ or $s$, but not both, may be equal to zero.
We assume that $H$ acts faithfully on $\bbG_a^- \times \bbG_a^-$, namely that no proper 
subgroup scheme of $H$ acts trivially on $\bbG_a^- \times \bbG_a^-$,
and we write $G$ for $(\bbG_a^- \times \bbG_a^-)\rtimes (\bbG_{a(r)} \times (\bbZ/p)^s)$.
We let $(\bbZ/p)^s=\langle g_1,\dots,g_s\rangle$, and write 
$t_i=g_i-1\in k(\bbZ/p)^s$, so that $\Delta(t_i)=t_i\otimes 1 + 1
\otimes t_i + t_i\otimes t_i$ ($1\le i\le s$). Since $H$ is unipotent,
the action of $H$ on $u$ and $v$ can be upper
triangularized. 
We choose $v$ to be the invariant element. 
Furthermore, there are enough automorphisms of
$\bbG_{a(r)}$ so that all faithful actions are equivalent. 
Thus there are constants
$\mu_i\in k$ such that the map $\tau \colon kH \otimes A \to A$
describing the action is given by
\begin{align*} 
\tau(s_1\otimes u) &= v, \\
\tau(s_i\otimes u) &= 0 & (2\le i \le r), \\
\tau(s_i\otimes v) &= 0 & (1\le i \le r), \\
\tau(t_i\otimes u) &= \mu_i v & (1\le i \le s), \\ 
\tau(t_i\otimes v) &=0 & (1\le i \le s).
\end{align*}

By abuse of notation, we write $u$ for $u\otimes 1$, $v$ for $v\otimes
1$, $s_i$ for $1 \otimes s_i$ and $t_i$ for $1\otimes t_i$ in $A\rtimes kH$.
These elements satisfy the following relations:
\begin{align*} 
u^2&=v^2=uv+vu=0, \\ 
s_1u&=us_1+v, \\
s_iu&=us_i+s_1^{p-1}\!\!\!\dots s_{i-1}^{p-1}v & (2\le i \le r), \\
s_iv&=vs_i & (1\le i \le r), \\
t_iv&=vt_i & (1\le i \le s), \\
t_iu&=ut_i+\mu_i v (1+t_i) & (1\le i\le s).
\end{align*}

\section{Steenrod operations}\label{se:Steenrod}

We shall need to use Steenrod operations in the cohomology of finite
supergroup schemes. The discussion of these in the literature is
almost, but not completely adequate for our purposes, 
and so we give a brief discussion here.

If $A$ is a $\bbZ$-graded cocommutative Hopf algebra over $\bbF_p$, the discussion in 
Section 11 of May~\cite{May:1970a} does the job. For $p$ odd, there
are natural operations 
\begin{align*}
\mcP^i\colon H^{s,t}(A,k) &\to H^{s+(2i-t)(p-1),pt}(A,k) \\
\beta\mcP^i \colon H^{s,t}(A,k) &\to H^{s+1+(2i-t)(p-1),pt}(A,k)
\end{align*}
satisfying, among others, the following properties:

\begin{enumerate}
\item[\rm (i)] $\mcP^i=0$ if either $2i<t$ or $2i>s+t$\newline
$\beta\mcP^i=0$ if either $2i<t$ or $2i\ge s+t$
\item[\rm (ii)] $\mcP^i(x)=x^p$ if $2i=s+t$
\item[\rm (iii)] $\mcP^j(xy) = \sum_{i}\mcP^i(x)\mcP^{j-i}(y)$\newline
$\beta\mcP^j(xy)=\sum_{i}(\beta\mcP^i(x)\mcP^{j-i}(y)+\mcP^i(x)\beta\mcP^{j-i}(y))$
\item[\rm (iv)] The $\mcP^i$ and $\beta\mcP^i$ satisfy the Adem relations.
\item[\rm (v)] $\mcP^i$ is $\bbF_p$-linear; that is, $\mcP^i(u+v)
  =\mcP^i(u)+\mcP^i(v)$, and $\mcP^i(\lambda u) = \lambda \mcP^i(u)$
  for $\lambda \in \bbF_p$ and $u,v\in H^{*,*}(A,k)$.
\end{enumerate}

For us, there are two problems with this. The first is that we want to
work over a more general field $k$ of characteristic $p$, not just
$\bbF_p$. As remarked by
Wilkerson~\cite{Wilkerson:1981a} (bottom of page~140), the only
difference is that the operations are no longer $k$-linear. Rather,
they are semilinear, so that (v) should be replaced by\smallskip

(v) $\mcP^i$ is $k$-semilinear; that is, $\mcP^i(u+v)
  =\mcP^i(u)+\mcP^i(v)$, and $\mcP^i(\lambda u) = \lambda^p \mcP^i(u)$
  for $\lambda \in k$ and $u,v\in H^{*,*}(A,k)$.\smallskip

The other problem is that if we wish to apply this to a $\bbZ/2$-graded
object, then the way the indices works involves subtracting an element
of $\bbZ/2$ from an element of $\bbZ$ and expecting an answer in
$\bbZ$. This clearly doesn't work, so we need to do some re-indexing
to take care of this problem. The origin of the problem is that May
has chosen to base the indexing of the operations on total degree rather than
internal degree. The rationale for doing this is that it avoids the
introduction of half-integer indexed operations, but the disadvantage
is that it only works for $\bbZ$-graded objects, and not for example
for $\bbZ/2$-graded objects.

In order to reindex using internal degree rather than total degree,
we rename May's $\mcP^i$ as our $\cP^{i-t/2}$.
Then we have
\begin{align*} 
\cP^i\colon H^{s,t}(A,k) &\to H^{s+2i(p-1),pt}(A,k) \\
\beta\cP^i\colon H^{s,t}(A,k) & \to H^{s+1+2i(p-1),pt}(A,k). 
\end{align*}
Here, $i\in \bbZ$ if $t$ is even and $i\in \bbZ+\half$ if $t$ is odd.
Note that since $p$ is odd,
$pt$ is equivalent to $t$ mod $2$, so the operations 
preserve internal degree as elements of $\bbZ/2$.

These operations are 
called $P^i$ in Theorem A1.5.2 of Appendix 1 in 
Ravenel \cite{Ravenel:1986a}.
They are called $\tilde \mcP^i$ in the discussion
following Theorem 11.8 of May~\cite{May:1970a}, but he 
ignores the operations indexed by $\bbZ+\half$.

The upshot of this reindexing is that at the expense 
of introducing half-integer indices for the Steenrod
operations, we have made the notation work for $\bbZ/2$-graded
objects. Properties (i) and (ii) have been reindexed, so that 
(i)--(v) are now as follows:
\begin{enumerate}
\item[\rm (i)] $\cP^i=0$ if either $i<0$ or $i>s/2$\newline
$\beta\cP^i=0$ if either $i<0$ or $i\ge s/2$
\item[\rm (ii)] $\cP^i(x)=x^p$ if $i=s/2$
\item[\rm (iii)] $\cP^j(xy) = \sum_{i}\cP^i(x)\cP^{j-i}(y)$\newline
$\beta\cP^j(xy)=\sum_{i}(\beta\cP^i(x)\cP^{j-i}(y)+\cP^i(x)\beta\cP^{j-i}(y))$
\item[\rm (iv)] The $\cP^i$ and $\beta\cP^i$ satisfy the Adem relations.
\item[\rm (v)] $\cP^i$ is $k$-semilinear; that is, $\cP^i(u+v)
  =\cP^i(u)+\cP^i(v)$, and $\cP^i(\lambda u) = \lambda^p \cP^i(u)$
  for $\lambda \in k$ and $u,v\in H^{*,*}(A,k)$.
\end{enumerate}

\begin{proposition}\label{pr:Steenrod-on-Ga^-}
The ring $H^{*,*}(\bbG_a^-,k)$ is a polynomial ring $k[\zeta]$ 
on a single generator $\zeta$ in degree $(1,1)$.
The action of the Steenrod operations on $H^{*,*}(\bbG_a^-,k)$ is given by
$\cP^\shalf(\zeta)=\zeta^p$,
$\beta\cP^\shalf(\zeta)=0$.
\end{proposition}
\begin{proof}
We prove this by reducing the grading modulo two on a $\bbZ$-graded
cocommutative Hopf algebra.
The cohomology of a 
$\bbZ$-graded Hopf algebra on a primitive exterior generator in degree one
is $k[\zeta]$ with $\zeta$ in degree $(1,1)$. 
If we compute the action of the Steenrod operations on this, 
the action of $\cP^\shalf=\mcP^1$ and $\beta\cP^\shalf=\beta\mcP^1$ 
follows from Theorem~11.8~(ii) of \cite{May:1970a}, and is given as in
the Proposition.
Now reduce the grading modulo two.
\end{proof}

We have
\begin{multline*} 
H^{*,*}(\bbG_a^-\times\bbG_{a(r)} \times (\bbZ/p)^s,k) = \\
k[\zeta]\otimes k[x_1,\dots,x_r] \otimes
\Lambda(\lambda_1,\dots,\lambda_r) \otimes
k[z_1,\dots,z_s] \otimes \Lambda(y_1,\dots,y_s). 
\end{multline*}
The degrees and action of the Steenrod operations are as follows.
\begin{center}
\renewcommand{\arraystretch}{1.4}
\begin{tabular}{|c|c|c|c|c|c|}
\hline
& \small degree & $\cP^{0^{\phantom{0}}}$ & $\beta\cP^0$ &
$\cP^\shalf$ & $\cP^1$   \\ 
\hline
$\zeta$ & $(1,1)$ & & & $\zeta^p$ &  \\
$\lambda_i$ & $(1,0)$ & $\lambda_{i+1}$ & $-x_i$ & & $0$  \\
$y_i$ & $(1,0)$ & $y_i$ & $z_i$ &  & $0$  \\
$x_i$ & $(2,0)$ & $x_{i+1}$ & $0$ & & $x_i^p$  \\
$z_i$ & $(2,0)$ & $z_i$ & $0$ & & $z_i^p$  \\ 
\hline
\end{tabular}\smallskip
\end{center}
Here, $\lambda_{i+1}$ and $x_{i+1}$ are taken to be zero if $i=r$.

\section{The local cohomology spectral sequence}
\label{se:local}

In this section we sketch the construction of a generalization 
of the local cohomology spectral sequence to finite supergroup 
schemes.  The spectral sequence was constructed by Benson and 
Carlson~\cite{Benson/Carlson:1994a} for finite groups; Greenlees 
gave a more robust construction in \cite{Greenlees:1995a}. The supergroup 
version comes with a twist which we now describe. 

Recall from Section~I.8 of Jantzen~\cite{Jantzen:1987a} that
there is a one dimensional representation $\delta_G$ of a finite group
scheme $G$, called the \emph{modular function}, and that by
Proposition~I.8.13 of \cite{Jantzen:1987a}, if $Q$ is a projective
$kG$-module then $\Soc(Q) \cong Q/\Rad(Q) \otimes \delta_G$.
This generalizes to finite supergroup schemes, without change in the
argument, the only extra feature being that $\delta_G$ comes with a
parity $\ep_G\in\bbZ/2$. So for example $\delta_{\bbG_a^-}$ is the trivial module, but
in odd internal degree, so we have $\ep_{\bbG_a^-}=1\in\bbZ/2$. 

The role of the modular function is that it appears in Tate duality which we 
deduce from the general statement of Auslander-Reiten duality. The latter gives an isomorphism 
\begin{equation} 
\label{eq:AR}
\Hom_k(\underline{\Hom}_G(M,N)) \cong \underline{\Hom}_G(N,\Omega \nu M)
\end{equation} 
(see 
~\cite[Proposition I.3.4]{Auslander:1978a}, also \cite[Corollary
p. 269]{Krause:2003a}). 
Here, 
\[\nu: \StMod G \to \StMod G\] 
is the Nakayama functor. For a 
finite supergroup schemes it is given by the formula 
\[ \nu(-) =  - \otimes_k \delta_G\] 
(\cite[Section 4]{Benson/Iyengar/Krause/Pevtsova:bikp4}).

Applying \eqref{eq:AR} to $N, \Omega^{n+1}M$, we get Tate duality for finite supergroup schemes:
\begin{align*} 
\widehat{\Ext}^{-n-1,*}(N,M) & \cong  \underline{\Hom} (N, \Omega^{n+1}M) \\
& \cong \Hom_k(\underline{\Hom}(\Omega^{n+1}M, \Omega \nu N),k) \\
& \cong  \Hom_k(\underline{\Hom}(\Omega^{n}M, N \otimes \delta_G),k) \\
& \cong  \Hom_k(\widehat{\Ext}^{n, * + \epsilon_G}(M, N \otimes \delta_G),k).
\end{align*}

In particular, for $M=N=k$, $n \geq 0$, this becomes 
\begin{equation} 
\label{eq:Tate} 
H_{n,j+\ep_G}(G,\delta_G) \cong \widehat H^{-n,j+\ep_G}(G,\delta_G) \cong \Hom_k(\widehat H^{n-1,* + \epsilon_G}(G, \delta_G),k). 
\end{equation}

The local cohomology spectral sequence is triply graded. The
gradings are firstly local cohomological, secondly group
cohomological, and thirdly internal parity. Repeating the constructions in 
\cite{Greenlees:1995a} or \cite{Benson:2001a} (explicitly, Section 3 in \cite{Benson:2001a}) verbatim up to the point where local duality is used to identify negative Tate cohomology and homology gives the following.

\begin{theorem}\label{th:local-coh}
	Let $G$ be a finite supergroup scheme. Then there is a spectral sequence
		\[ E_2^{s,t,*} = H^{s,t}_\mfm H^{*,*}(G,k) \Rightarrow
	\widehat H^{s+t-1<0,*}(G,k)\]
	converging to the negative part of Tate cohomology. 
\end{theorem}

Now applying the duality isomorphism \eqref{eq:Tate}, we produce the local cohomology spectral sequence converging to homology.

\begin{corollary}
	\label{cor:local-coh} Let $G$ be a finite supergroup scheme. Then there is a local
	cohomology spectral sequence
		\[ E_2^{s,t,j} = H^{s,t,j}_\mfm H^{*,*}(G,k) \Rightarrow
	H_{-s-t,j+\ep_G}(G, \delta_G).\]
\end{corollary}

\begin{definition}
	A finite supergroup scheme is called {\it unimodular} if the modular function $\delta_G$ is the trivial module in degree $\ep_G$.
\end{definition}

\begin{remark} Finite unipotent supergroup schemes are unimodular since the only one-dimensional representations are given by the trivial module $k$ in either even or odd degree. 
	\end{remark}

Let $r$ be the Krull dimension of $H^{*,*}(G,k)$. Then $H^{i,*}_\mfm H^{*,*}(G,k) =0$ for $i>r$. Hence, there is an edge homomorphism of the local cohomology spectral sequence:
\begin{equation} 
\label{eq:edge} 
H^{r,t,j}_\mfm H^{*,*}(G,k) \to 
H_{-r-t,j+\ep_G}(G, \delta_G). 
\end{equation}

We wish to use the following consequences of Corollary~\ref{cor:local-coh}. The statement of the first  Corollary~\ref{co:Gorenstein} is a direct analogue of 
\cite[Corollary 2.3]{Greenlees:1995a} (see also \cite{Benson/Carlson:1994a}). 

\begin{corollary}\label{co:Gorenstein}
Let $G$ be a unimodular finite supergroup scheme. 
If $H^{*,*}(G,k)$ is Cohen--Macaulay,
then it is Gorenstein, with shift $(0,\ep_G)$. 

\end{corollary}

\begin{proof}
	Since $G$ is unimodular, $\delta_G$ is a trivial module. The Cohen--Macaulay assumption on $H^{*,*}(G,k)$ implies that the edge map of \eqref{eq:edge} is an isomorphism which therefore identifies the top local cohomology $H^r_\mfm(H^{*,*}(G,k))$ with homology $H_{*,*+\epsilon_G}(G,k)$. This is linear dual to cohomology $H^{*, *+\epsilon_G}(G,k)$, and, hence, the top local cohomology module is the injective hull of the trivial $H^{*,*}(G,k)$-module $k$ in the internal degree $\epsilon_G$.   Hence, $H^{*,*}(G,k)$ is Gorenstein (see \cite[Theorem 11.26]{Iyengar:2008a} or \cite[Theorem 1.3.4]{Goto/Watanabe:1978a} where the graded case is made explicit).
\end{proof}

Recall that $\{ \zeta_1,\ldots,\zeta_r \}$ is a {\it system of parameters} of a (graded) commutative $k$-algebra $A$ if $k[\zeta_1,\ldots,\zeta_r] \subset A$ is a Noether normalization of $A$, that  is, $A$
is a finite module over $k[\zeta_1,\ldots,\zeta_r]$.  The last corollary is a general property of graded Gorenstein $k$-algebras.   

\begin{corollary}\label{co:Poincare}
	Let $G$ be a unimodular finite supergroup scheme. 
	Assume $H^{*,*}(G,k)$ is Cohen--Macaulay, and 
		let $\zeta_1,\dots,\zeta_r$ be a regular homogeneous sequence of 
	parameters in $H^{*,*}(G,k)$.  Then
	the quotient $H^{*,*}(G,k)/(\zeta_1,\dots,\zeta_r)$ is a finite 
	Poincar\'e duality algebra with dualizing degree $(-r,\ep_G)+\sum_{i=1}^r |\zeta_i|$.
\end{corollary}

\section{The case $(\bbG_a^- \times \bbG_a^-)\rtimes \bbG_{a(1)}$}%
\label{se:Ga(1)}

Let $G=(\bbG_a^- \times \bbG_a^-)\rtimes \bbG_{a(1)}$. This is a finite supergroup scheme of height one and, 
hence, $kG$ is isomorphic to the restricted universal enveloping algebra of the three dimensional Lie superalgebra
$\mathfrak g$ (see, for example, \cite[Lemma 4.4.2]{Drupieski:2013a}). 
The Lie superalgebra $\mathfrak g$ has a basis consisting of odd elements $u$ and $v$, and an even element $t$. 
Specializing calculations in Section~\ref{se:semidirect} to this case, we get that the Lie algebra generators satisfy 
the following relations 
\[
[u,v]=0,\quad
 [t,v]=0, \quad
 [t,u]=v
\]
where $[\ \,,\ ]$ is the supercommutator in $\mathfrak g$. Thus $kG$ has the following 
presentation: 
\begin{equation}
\label{eq:Ga1}
kG = \frac{k[u,v, t]}{(u^2, v^2, uv+vu, t^p, tv-vt, tu-ut-v)}. 
\end{equation} 

\begin{theorem}\label{th:Ga1}
Let $G= (\bbG_a^-\times\bbG_a^-)\rtimes \bbG_{a(1)}$, with $\bbG_{a(1)}$ acting non-trivially.
Then $H^{*,*}(G,k)$ is generated by 
\[ \zeta\in H^{1,1}(G,k),\ x\in H^{2,0}(G,k), \ \kappa\in
  H^{p,1}(G,k),\ \lambda_i\in H^{i,1+i}(G,k)\ (1\le i \le p-2) \] 
with the relations
\begin{flalign*} 
\lambda_i\zeta=0\ (1\le i\le p-2), \\ 
x\zeta^{p-1}=0, \\
\lambda_i\lambda_j = 0 \text{ for } i+j \not = p, \\
\lambda_i\lambda_{p-i} = \alpha_i x\zeta^{p-2} \text{ where } \alpha_i \not = 0
\end{flalign*}

Then the Poincar\'e series is given by
\[ \sum_n t^n \dim_k H^{n,*}(G,k) = 1/(1-t)^2. \]
\end{theorem}

\begin{proof} 
We examine two spectral sequences, the first one given by the semidirect product: 
\[ H^i(\bbG_{a(1)},H^{j,*}(\bbG_a^- \times \bbG_a^-, k)) \Rightarrow H^{i+j,*}(G,k). \]
Let $V = ku \oplus kv$ be the two dimensional supervector 
space generated by $u,v$.
We have 
\[H^{*,*}(\bbG_a^-\times \bbG_a^-,k)\cong S^{*,*}(V^\#) = k[\zeta,\eta]
\] with $\zeta$
and $\eta$ in degree $(1,1)$, dual to the generators $u,v$. 
For $0\le j\le p-1$,  
\[ H^{j,*}(\bbG_a^-\times \bbG_a^-,k) \cong S^j(V^\#) \] is an
indecomposable $k\bbG_{a(1)}$-module of length $j+1$. 
It is projective for $j=p-1$, and not otherwise. 
Hence, we have the following restrictions on dimensions 
of the $E^2$ term of the spectral sequence: 
\begin{align}
\label{eq:dim}
\dim H^i(\bbG_{a(1)},H^{j,*}(\bbG_a^-\times \bbG_a^-,k)) = 1 & \text{ for } 0\le j\le p-2,\\
\dim H^1(\bbG_{a(1)},H^{p-1,*}(\bbG_a^-\times \bbG_a^-,k)) = 0, & \\
\dim H^0(\bbG_{a(1)},H^{p,*}(\bbG_a^-\times \bbG_a^-,k)) = 2. &
\end{align}
To justify the last equality, we do a calculation:
\[H^0(\bbG_{a(1)},H^{p}(\bbG_a^-\times \bbG_a^-,k)) = H^0(\bbG_{a(1)}, S^p(V^\#))  = k\zeta^p \oplus k \eta^p
\]
where the last equality is a special case of Lemma~\ref{le:Gar}. 

We conclude that 
\begin{equation}
\label{eq:dim2} 
\dim H^n(G,k) \leq \sum\limits_{i+j=n} \dim H^i(\bbG_{a(1)},H^j(\bbG_a^-\times \bbG_a^-,k)) = n+1
\end{equation}
for $0 \leq n \leq p$.

We now examine the spectral sequence 
\[ H^{*,*}(\bbG_a^-\times
\bbG_{a(1)},H^{*,*}(\bbG_a^-,k)) \Rightarrow H^{*,*}(G,k) \]
corresponding to the central extension
\[ 1 \to \bbG_a^- \to G \to \bbG_a^- \times \bbG_{a(1)} \to 1\]
We write 
\begin{align*} 
H^{*,*}(\bbG_a^-\times\bbG_{a(1)},H^{*,*}(\bbG_a^-,k))
&= H^{*,*}(\bbG_a^-,k) \otimes H^{*,0}(\bbG_{a(1)},k) \otimes H^{*,*}(\bbG_a^-,k) \\
&= k[\zeta,x] \otimes \Lambda(\lambda)
\otimes k[\eta] 
\end{align*} 
with $\zeta$ the generator of the first $H^{*,*}(\bbG_a^-,k)$, 
$ x, \lambda$ the generators of $H^{*,0}(\bbG_{a(1)},k)$, and $\eta$ the 
generator of the second $H^{*,*}(\bbG_a^-,k)$. The degrees of 
the generators in the spectral sequence are as follows: 
\[ |\zeta|=(1,0,1),\quad |\lambda|=(1,0,0),\quad |x|=(2,0,0),\quad
|\eta|=(0,1,1).\] 
Here, the first two indices are the horizontal and
vertical directions in the spectral sequence, and the third is the
$\bbZ/2$-grading.

{\tiny
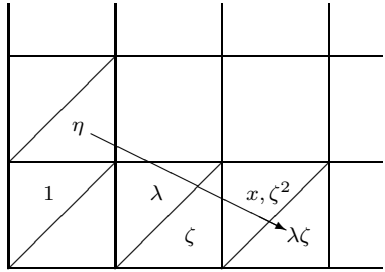
\begin{center}
\begin{picture}(160,120)(-10,-10)
\put(0,0){\line(1,0){140}}
\put(0,0){\line(0,1){100}}
\put(0,40){\line(1,0){140}}
\put(0,80){\line(1,0){140}}
\put(40,0){\line(0,1){100}}
\put(80,0){\line(0,1){100}}
\put(120,0){\line(0,1){100}}
\put(0,0){\line(1,1){40}}
\put(40,0){\line(1,1){40}}
\put(80,0){\line(1,1){40}}
\put(0,40){\line(1,1){40}}
\put(13,26){$1$}
\put(53,26){$\lambda$}
\put(66,10){$\zeta$}
\put(89,26){$x,\zeta^2$}
\put(104,10){$\lambda\zeta$}
\put(24,52){$\eta$}
\put(31,51){\vector(2,-1){73}}
\end{picture}
\end{center}
}

Let $\mu_p =  \bbG_{m(1)}$ be the finite group scheme of $p^{\rm th}$ roots of unity. 
Then $\mu_p \times \mu_p$  
acts on $kG$ (given by the presentation in \eqref{eq:Ga1})
in such a way that the first copy is acting on $u$ and the second is
acting on $t$. Both copies act on the commutator $v$.
Each monomial in the $E_2$ page of this spectral sequence is then an
eigenvector of $\mu_p \times \mu_p$. The weights are
elements of $\bbZ/(p-1) \times \bbZ/(p-1)$, and are given by 
\begin{align*}
\Vert\zeta\Vert=(1,0), \\ 
\Vert\lambda\Vert = (0,1), \\  
\Vert x\Vert = (0,1), \\ 
\Vert\eta\Vert=(1,1). 
\end{align*}

The differentials
in the spectral sequence have to preserve both the weight and the 
$\bbZ/2$-grading. The latter implies that $x$, $\zeta^2$ cannot be hit by $d_2(\eta)$ and, 
hence, survive to $E_\infty$.  Since $\dim H^{1,*}(G,k) \leq 2$ by \eqref{eq:dim},  we conclude that $\eta$ must die in $E_\infty$, and, hence, 
$d_2(\eta)=\lambda\zeta$. By the Newton--Leibniz rule, we get that a monomial 
$\lambda^\epsilon \eta^a \zeta^bx^c$ does not survive in $E^3$ if 
\begin{equation} 
\label{eq:e3}
\{\epsilon =1, a \leq p-1, 1 \leq b\}
\end{equation}
in which case it is not in the kernel of $d_2$ or if 
\begin{equation} 
\label{eq:e4}
\{\epsilon =0  \text{ and } 1 \leq  a \leq p-1\},
\end{equation}
in which case it is in the image of $d_2$. On the  other hand, $d_2(\eta^p)=0$. 

We conclude that  the $E_3$ page is generated by the permanent cycles  $\lambda$, $\zeta$ and $x$
on the base, the element $\eta^p$ on the fibre, and
$\lambda\eta,\lambda\eta^2,\dots,\lambda\eta^{p-1}$ in the first column. Moreover, $E_3$ has the relations
\begin{equation}
\label{eq:rel} 
(\lambda \eta^i) \zeta =0, \quad (\lambda \eta^i)(\lambda \eta^j) =0
\end{equation}
for $1 \leq i,j \leq p-1$. Since $\cP^\shalf(\eta)=\eta^p$ and $\cP^\shalf(\lambda\zeta)=0$, Kudo's
transgression theorem (\cite[Theorem 3.4]{May:1970a}) implies that $\eta^p$ survives to the $E_\infty$
page of the spectral sequence.

{\tiny
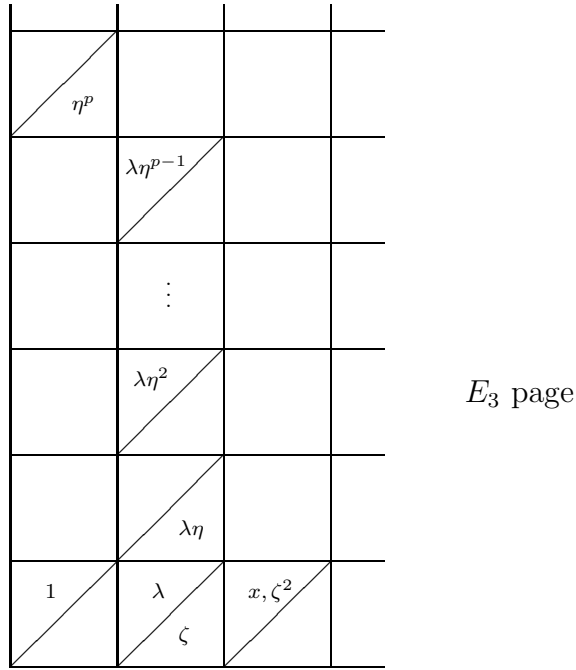
\begin{center}
\begin{picture}(160,270)(-10,-10)
\put(0,0){\line(1,0){140}}
\put(0,0){\line(0,1){250}}
\put(0,40){\line(1,0){140}}
\put(0,80){\line(1,0){140}}
\put(0,120){\line(1,0){140}}
\put(0,160){\line(1,0){140}}
\put(0,200){\line(1,0){140}}
\put(0,240){\line(1,0){140}}
\put(40,0){\line(0,1){250}}
\put(80,0){\line(0,1){250}}
\put(120,0){\line(0,1){250}}
\put(0,0){\line(1,1){40}}
\put(40,0){\line(1,1){40}}
\put(80,0){\line(1,1){40}}
\put(40,40){\line(1,1){40}}
\put(40,80){\line(1,1){40}}
\put(40,160){\line(1,1){40}}
\put(0,200){\line(1,1){40}}
\put(13,26){$1$}
\put(53,26){$\lambda$}
\put(63,10){$\zeta$}
\put(89,26){$x,\zeta^2$}
\put(63,50){$\lambda\eta$}
\put(46,106){$\lambda\eta^2$}
\put(58,136){$\vdots$}
\put(43,186){$\lambda\eta^{p-1}$}
\put(23,210){$\eta^p$}
\put(170, 100){\normalsize $E_3$ page}
\end{picture}
\end{center}
}

There remains the question of the values of the differentials
$d_3,\dots,d_p$ on the elements
$\lambda\eta,\dots,\lambda\eta^{p-1}$. 

\begin{claim}\label{cl:d}
The differentials $d_3,\dots,d_{p-1}$ vanish on the elements $\lambda\eta,\dots,\lambda\eta^{p-1}$. 
\end{claim} 
\begin{proof}[Proof of Claim]  Suppose some differential $d_\ell$ is non trivial on $\lambda \eta^i$ and let 
$\lambda^\ep \eta^{i_1} x^{i_2} \zeta^{i_3}$ be in the target of that differential. If $i_1 \not = 0$, then \eqref{eq:e3}, \eqref{eq:e4}  
imply that $i_3 =0$ and $\ep =1$. Hence $\lambda \eta^i$ hits a monomial of the form $\lambda \eta^{i_1}
x^{i_2}$.  The weights of these monomials are $(i, i+1)$ and $(i_1, 1 + i_1 + i_2)$ respectively. Since the weights 
are preserved, we conclude $i = i_1$, which contradicts the fact that $d_\ell$ must lower the exponent of 
$\eta$ by $\ell -1$. 

Therefore, $i_1=0$, and the differential $d_\ell$ on $\lambda \eta^i x^j$ 
hits something on the base, a monomial of the form $\lambda^\ep x^{i_2} \zeta^{i_3}$.  The weights are $(i, i+1)$ and 
$(i_3, \ep + i_2)$ respectively. Hence, $ i_3 = i > 0$. By  \eqref{eq:e3}, \eqref{eq:e4}, $\ep=0$. The conditions on the second weight and 
the total degree now give the following equations: 
\begin{align*}
1+i&\equiv i_2 &\pmod{p-1} \\
1+i& = 2i_2+i -1,
\end{align*}
The only solution is $i = p-1$, $i_2=1$, that is, the only possible non trivial differential is $d_p(\lambda \eta^{p-1})$. 
This proves the claim. 
\end{proof}

Claim~\ref{cl:d} immediately implies that $\lambda\eta,\dots,\lambda\eta^{p-2}$ are (non-trivial) permanent cycles. We 
also conclude that all differentials up to $d_{p-1}$ vanish on all generators of $E_3$. Hence, $E_3=E_p$.  It remains 
to determine the differential $d_p$ on $\lambda\eta^{p-1}$.  

\begin{claim} \label{cl:dp} 
$d_p(\lambda\eta^{p-1}) = \alpha x\zeta^{p-1}$
with $\alpha \not = 0$.
\end{claim}
\begin{proof}[Proof of Claim] We have $\dim H^{p,*}(G,k) \leq p+1$ by \eqref{eq:dim}. On the other hand,  we 
established at least $p+1$ linearly independent cycles of total degree $p$ in $E_\infty$: 
\[\{\eta^p, \lambda x^\frac{p-1}{2}, \lambda \eta^2 x ^\frac{p-3}{2},  \ldots, 
\lambda \eta^{p-3}x, \zeta^p, x \zeta^{p-2}, \ldots, x^\frac{p-1}{2}\zeta\}.\] 
Hence, $\lambda \eta^{p-1}$ is not a permanent cycle, since otherwise
we would have $\dim H^{p,*}(G,k)\ge p+2$.
We have already computed that $d_p(\lambda\eta^{p-1})$ is some multiple of
$x\zeta^{p-1}$. This proves the claim.  
\end{proof}

This completes the determination of the $E_\infty$ page of the spectral sequence 
of the central extension. We also conclude that $x\zeta^{p-1}$ is zero in $H^{p+1,0}(G,k)$. 

To describe the cohomology ring $H^{*,*}(G,k)$, we start by giving
elements of $E_\infty^{*,0}$ the same names in $H^{*,*}(G,k)$. We
choose a representative $\kappa\in H^{p,1}(G,k)$ of $\eta^p\in
E_\infty^{0,p}$; this is a non zero-divisor.
Next, choose $\lambda_2,\dots,\lambda_{p-1}$ to 
be representatives in $H^{*,*}(G,k)$ of the elements $\lambda\eta,\dots,\lambda\eta^{p-2}$ in $E_\infty$,
as follows. Arguing with congruences as before, we see that there is only one dimension
in each of these degrees with the correct weight for the 
action of $\mu_p\times\mu_p$, so this gives a well
defined representative. We also write $\lambda_1$ for $\lambda$.

Using weights and congruences, 
we see that the product $\lambda_i\zeta$ is equal to zero. Similarly,
$\lambda_i\lambda_j$ is either zero or a multiple of
$x\zeta^{p-2}$, and the latter can only happen when $i+j= p$.
In the case where $i+j=p$, we claim that $\lambda_i\lambda_j$ is a
non-zero multiple of $x\zeta^{p-2}$. The proof of
this claim uses the local cohomology spectral sequence in the form of
Corollary~\ref{co:Poincare}, and this will complete the computation
of $H^{*,*}(G,k)$.

\begin{claim}
$H^{*,*}(G,k)$ is Cohen--Macaulay. A regular homogeneous sequence of
parameters is given by $\kappa\in H^{p,1}(G,k)$ and $x+\zeta^2\in H^{2,0}(G,k)$.
\end{claim}
\begin{proof}[Proof of Claim]
It suffices to show that $\eta^p$ and $x+\zeta^2$ form a regular
sequence in $E_\infty$. Since $\eta^p$ is a non zero-divisor, this
amounts to showing that $x+\zeta^2$ is a non zero-divisor on $E_\infty/(\eta^p)$.
The non-zero monomials in $E_\infty/(\eta^p)$ come in two types. The first
are the $\zeta^i x^j$ with $j=0$ if $i\ge p-1$.
Multiplying by $x+\zeta^2$, these go to
$\zeta^{i+2}x^j+\zeta^i x^{j+1}$, where the second term is zero if
$i\ge p-1$ and $j=0$. Ordering
lexicographically in $(i,j)$, we see that these are linearly
independent,
because their leading terms are linearly independent.
The second kind of monomials  are the $\lambda_i x^j$
with $1\le i\le p-1$. These go to $\lambda_i x^{j+1}$, which are again
linearly independent.
\end{proof}

We are now in a position to complete the proof of
Theorem~\ref{th:Ga1}. Since $H^{*,*}(G,k)$ is Cohen-Macaulay, Corollary~\ref{co:Poincare}
implies that $H^{*,*}(G,k)/(\kappa,x+\zeta^2)$ has Poincare duality with dualizing element 
in degree $p$.  The ring $H^{*,*}(G,k)/(\kappa,x+\zeta^2)$ has a
basis consisting of $\zeta^i\in H^{i,i}(G,k)$ with $0\le i \le p$ and 
$\lambda_i\in H^{i,i-1}(G,k)$ with $1\le i\le p-1$ (where the second degree is taken mod 2). 
The top degree dualizing element is $\zeta^p$, which is equivalent modulo $x+\zeta^2$
to $x\zeta^{p-2}$. For each element in
$H^{*,*}(G,k)/(\kappa,x+\zeta^2)$ there has to be an element whose
product with it is equal to the dualizing element. Applying this to
$\lambda_i$, we see that $\lambda_i\lambda_{p-i}$ has to be non-zero,
and is therefore a non-zero multiple of $x\zeta^{p-2}$.
Replacing some of the
$\lambda_i$ by non-zero multiples, we have
$\lambda_i\lambda_{p-i}=x\zeta^{p-2}$.
We now have all the generators and relations for the ring structure
on $H^{*,*}(G,k)$, completing the proof of Theorem~\ref{th:Ga1}. 
\end{proof}

\begin{corollary}
With $G$ as in Theorem~\ref{th:Ga1}, we have
\begin{equation*} 
\sum_{n \ge 0}t^n\dim_kH^{n,*}(G,k) = 1/(1-t)^2. 
\end{equation*}
\end{corollary}
\begin{proof}
This is an easy dimension count using the theorem.
\end{proof}

To analyze the case of a more general semidirect product as we do in Section~\ref{se:proof}, 
we don't need the force of Theorem~\ref{th:Ga1} but only a particular calculation which was obtained 
as part of the proof. 
\begin{corollary}[of the proof]
\label{co:calculation} In the notation of the proof of Theorem~\ref{th:Ga1}, we have that $d_p(\lambda \eta^{p-1})$ 
is a non-zero multiple of $x\zeta^{p-1}$. 
\end{corollary}

\begin{remark}
\label{re:Zp}
The group algebra of the semidirect product $(\bbG_a^- \times \bbG_a^-)\rtimes \bbZ/p$ is
generated by elements $u$, $v$ and $g$ satisfying $u^2=0$, $v^2=0$, $g^p=1$,
$uv+vu=0$, $gu=(u+v)g$, $gv=vg$. Writing $t$ for $g-1$, this becomes
\[ u^2=v^2=uv+vu=t^p=0,\quad tv=vt,\quad tu=ut+v+vt. \]
Substituting $v'=v+vt$ then gives the presentation of the group
algebra studied in this section.
Since the cohomology only depends on the algebra structure, not on
the comultiplication, we get the same answer as in the case of 
$(\bbG_a^- \times \bbG_a^-)\rtimes \bbG_{a(1)}$ computed in this section.
\end{remark}


\section{An invariant theory computation}

Let $H=\bbG_{a(r)} \times (\bbZ/p)^s$, acting on $\bbG_a^-\times \bbG_a^-$
as in Section \ref{se:semidirect}, and let $G$ be the semidirect
product. In preparation for the computation of $H^{*,*}(G,k)$, we
begin with an invariant theory computation.

We have $H^{*,*}(\bbG_a^- \times \bbG_a^-,k) \cong k[X,Y]$
where $X$ and $Y$ are in degree $(1,1)$. We choose the notation so
that $Y$ is fixed by this action, and $X$ is sent to $X$ plus multiples of
$Y$. In this section, we compute
the invariants of such an action. To this end, we consider $k[X,Y]$ to
be the ring of polynomial functions on the vector space $V$ with basis 
$u$ and $v$, so that $Y$ and $X$ form the dual basis of the linear
functions on $V$.

We begin with the case $s=0$, namely $H=\bbG_{a(r)}$.
In general, an action of a group scheme $G$ on a scheme $Z$
over a scheme $S$, is given
by a map $G \times_S Z \to Z$ satisfying the usual associative law
defining an action. Corresponding to this is a map of coordinate rings
$k[Z] \to k[G] \otimes_{k[S]} k[Z]$ giving the coaction of $k[G]$ on
$k[Z]$.  Then the fixed points  $k[Z]^G$ is the subring of $k[Z]$
consisting of those $f$ whose image in $k[G] \otimes_{k[S]} k[Z]$ 
under the comodule maps is equal to $1 \otimes f$.

In our case, we have $k[\bbG_{a(r)}]=k[t]/(t^{p^r})$ with $t$ a
primitive element in the Hopf structure. 
The action $\bbG_{a(r)}$ on 
$V$ corresponds to a map $\bbG_{a(r)} \times_{\Spec k} V \to V$,
and then to a map of coordinate rings 
$k[X,Y] \to k[t]/(t^{p^r}) \otimes k[X,Y]$. 
The fact that $Y$ is fixed by the action implies that $Y$ maps to
$1\otimes Y$. The fact that $X$ is sent to $X$ plus multiples of $Y$,
together with the identities describing a coaction,
imply that $X$ maps to an element of the form $f(t) \otimes Y +
1 \otimes X$ where $f$ is a linear combination of the $t^{p^i}$ with
$0\le i < r$. Faithfulness of the action then implies that the term
with $i=0$ is non-zero. Thus $f(t)$ is primitive, and 
there is an automorphism of $\bbG_{a(r)}$ 
sending $f(t)$ to $t$. So without loss of generality,
the action is given by $X \mapsto t\otimes Y + 1 \otimes X$.

\begin{lemma}\label{le:Gar}
The invariants of the action of $\bbG_{a(r)}$ on $k[X,Y]$ 
are given by 
\begin{equation*}
k[X,Y]^{\bbG_{a(r)}}=k[X^{p^r},Y].
\end{equation*}
\end{lemma}
\begin{proof}
This is an easy computation.
\end{proof}

Next we describe the case $r=0$, namely $H=\langle
g_1,\dots,g_s\rangle\cong (\bbZ/p)^s$ with the $g_i$ commuting
elements of order $p$. In this case, the action again fixes $Y$, and
we have $g_i(X)=X-\mu_iY$  ($1\le i \le s$).
The fact that the action is faithful is equivalent to the statement
that the field elements $\mu_i$ are linearly independent over the
ground field $\bbF_p$. 
Then the orbit product 
\[ \phi(X,Y)= \prod_{g\in(\bbZ/p)^s} g(X) = 
\prod_{(a_1,\dots,a_s)\in(\bbF_p)^s}X+(a_1\mu_1+\dots+a_s\mu_s)Y \]
is clearly an invariant.  

\begin{lemma}\label{le:Zps}
The invariants of $(\bbZ/p)^s$ on $k[X,Y]$ are given by 
\[ k[X,Y]^{(\bbZ/p)^s}=k[\phi(X,Y),Y], \] 
where $\phi(X,Y)$ is given above.
\end{lemma}
\begin{proof}
See for example Proposition 2.2 of Campbell, Shank and Wehlau \cite{Campbell/Shank/Wehlau:2013a}.
\end{proof}

Putting these together, we have the following theorem.

\begin{theorem}\label{th:invariants}
The invariants of $\bbG_{a(r)}\times (\bbZ/p)^s$ 
on $k[X,Y]$ are given by
\[ k[X,Y]^{\bbG_{a(r)}\times (\bbZ/p)^s}=k[\phi(X,Y)^{p^r},Y]. \]
\end{theorem}
\begin{proof}
This follows by applying first Lemma \ref{le:Zps} and then Lemma \ref{le:Gar}.
\end{proof}

\section{Structure of symmetric powers}\label{se:sympowers}

We can use the computation of the last section to help us understand
the structure of the polynomial functions on the two dimensional space
$V$, as a module for $H=\bbG_{a(r)} \times (\bbZ/p)^s$. Note that the
space of polynomials of degree $n$ is $S^n(V^\#)$, and has a basis
consisting of the monomials $X^iY^{n-i}$ for $0\le i \le n$. In
particular, the dimension of $S^n(V^\#)$ is $n+1$.

\begin{lemma}\label{le:fixedpoints}
Let $M$ be a $kH$-module whose fixed points $M^H$ are one
dimensional. Then $M$ is indecomposable and 
$\dim_k(M)\le p^{r+s}$, with equality if and only
if $M$ is projective.
\end{lemma}
\begin{proof}
Since $H$ is unipotent, $kH$ is a local self-injective algebra. So if
$M^H$ is one dimensional, then the injective hull of $M$ is $kH$. 
Since $kH$ has dimension $p^{r+s}$, the lemma follows.
\end{proof}

\begin{theorem}\label{th:sympowers-low}
For $n < p^{r+s}-1$, the symmetric $n$th power $S^n(V^\#)$ is a
non-projective indecomposable $kH$-module. 
The module 
$S^{p^{r+s}-1}(V^\#)$ is a free $kH$-module of rank one.
\end{theorem}
\begin{proof}
It follows from Theorem \ref{th:invariants} that $S^n(V^\#)^H$ is one
dimensional for $n\le p^{r+s}-1$. The theorem therefore follows
from Lemma \ref{le:fixedpoints}.
\end{proof}

\begin{definition}
Let $f(X,Y)=\sum_{i=0}^n a_iX^iY^{n-i}$ be a degree $n$ 
homogeneous polynomial in $X$ and $Y$. Then the
\emph{leading term} of $f$ is the term $a_iX^iY^{n-i}$ for the
largest value of $i$ with $a_i\ne 0$.
\end{definition}

\begin{theorem}
For $n\ge p^{r+s}$, we have $S^n(V^\#) \cong kH \oplus
S^{n-p^{r+s}}(V^\#)$.
\end{theorem}
\begin{proof}
Consider the map 
$S^{p^{r+s}-1}(V^\#) \to S^n(V^\#)$ given by multiplication by
$Y^{n+1-p^{r+s}}$, and
the map $S^{n-p^{r+s}}(V^\#)\to S^n(V^\#)$ given by
multiplication by $\phi(X,Y)$. 
Examining the leading terms of the images of monomials
under these maps, we see that these maps are injective, 
the images span and intersect in
zero. Therefore $S^n(V^\#)$  is an internal direct sum of 
$Y^{n+1-p^{r+s}}.S^{p^{r+s}-1}(V^\#)$
and $\phi(X,Y).S^{n-p^{r+s}}(V^\#)$. By Theorem \ref{th:sympowers-low},
the first summand is isomorphic to $kH$.
\end{proof}

\begin{corollary}
The $kH$-module $S^n(V^\#)$ is projective if and only if $n$ is
congruent to $-1$ modulo $p^{r+s}$.\qed
\end{corollary}

Next, we examine the modules $S^{p^i-1}(V^\#)$ with $1\le i < r+s$. 
We have seen that these modules are not projective, but we shall
show that the complexity is exactly $r+s-i$, and we shall identify 
the annihilator of cohomology. The method we use is a variation of
the Steinberg tensor product theorem.

\begin{lemma}\label{le:hyperplane}
The $kH$-module $S^{p-1}(V^\#)$
is a uniserial module whose rank variety
is the hyperplane consisting of the points
$(\gamma_1,\dots,\gamma_r,\alpha_1,\dots,\alpha_s)\in \bbA^{r+s}(k)$ such that
\[ -\gamma_1+\alpha_1\mu_1+\dots+\alpha_s\mu_s=0. \]
\end{lemma}
\begin{proof}
We have
\begin{align*}
s_1(X^i)&=iX^{i-1}Y \\
s_j(X^i)&= 0 & 2\le j \le r \\
(g_j-1)(X^i)&=(X-\mu_j Y)^i-X^i =-i\mu_jX^{i-1}Y+\cdots & 1\le j \le s 
\end{align*}
and so if $(\gamma_1,\dots,\gamma_r,\alpha_1,\dots,\alpha_s)\in \bbA^{r+s}(k) \setminus \{0\}$
then
\begin{multline*} 
(\gamma_1s_1+\dots+\gamma_rs_r+\alpha_1(g_1-1)+\dots
+\alpha_s(g_s-1))(X^{p-1})\\
=(-\gamma_1+\alpha_1\mu_1+\dots+\alpha_s\mu_s)X^{p-2}Y + \cdots 
\end{multline*}
Continuing this way, we have
\begin{multline*} 
(\gamma_1s_1+\dots+\gamma_rs_r+\alpha_1(g_1-1)+\dots +
\alpha_s(g_s-1))^i(X^{p-1})\\
=i!(-\gamma_1+\alpha_1\mu_1+\dots+\alpha_s\mu_s)^i X^{p-1-i}Y^i +
\cdots 
\end{multline*}
and finally
\begin{multline*} 
(\gamma_1s_1+\dots+\gamma_rs_r+\alpha_1(g_1-1)+\dots +
\alpha_s(g_s-1))^{p-1}(X^{p-1})\\
=-(-\gamma_1+\alpha_1\mu_1+\dots+\alpha_s\mu_s)^{p-1} Y^{p-1}. 
\end{multline*}
So the restriction to the shifted subgroup defined by
$(\gamma_1,\dots,\gamma_r,\alpha_1,\dots,\alpha_s)$ is projective if and only if 
$-\gamma_1+\alpha_1\mu_1+\dots+\alpha_s\mu_s\ne 0$.

Since there is a non-trivial shifted subgroup such that the
restriction is projective, it follows that the module is uniserial.
\end{proof}

\begin{lemma}\label{le:Steinberg}
For $1\le i \le r+s$ the $kH$-module $S^{p^i-1}(V^\#)$ is isomorphic to
the tensor product of Frobenius twists
\[ S^{p-1}(V^\#) \otimes S^{p-1}(V^\#)^{(1)} \otimes \dots \otimes
  S^{p-1}(V^\#)^{(i-1)}. \]
\end{lemma}
\begin{proof}
We regard $S^{p-1}(V^\#)^{(j)}$ as the linear span of the $p^j$th
powers of the elements of $S^{p-1}(V^\#)$. Examining monomials, it is
apparent that multiplication provides the required isomorphism from
the tensor product to $S^{p^i-1}(V^\#)$.
\end{proof}

\begin{theorem}\label{th:rank-variety}
For $1\le i \le r+s$ the rank variety of the module $S^{p^i-1}(V^\#)$
is the linear subspace of $\bbA^{r+s}$ defined by the first $i$ rows
of the $(r+s) \times (r+s)$ matrix
\[ \renewcommand{\arraystretch}{1.2}
\left(\begin{array}{rrcr|ccc} -1 & 0 & \cdots & 0 & \mu_1 & \dots & \mu_s \\
0 & -1 & & 0 & \mu_1^p &  & \mu_s^p \\
0 & 0 & & 0 & \mu_1^{p^2} & & \mu_s^{p^2} \\ 
\vdots\, & & & \vdots\, & \vdots & & \vdots \\
0 & 0 & & -1 & \mu_1^{p^{r-1}} & \cdots & \mu_s^{p^{r-1}} \\ \hline
0 & 0 & & 0 & \mu_1^{p^r} & & \mu_s^{p^r} \\
\vdots\, & & & \vdots\, & \vdots & & \vdots \\
0 & 0 & \cdots & 0 & \mu_1^{p^{r+s-1}}\! & & \mu_s^{p^{r+s-1}}\!
\end{array}\right) \]
The rows of this matrix are linearly independent, so the complexity of
$S^{p^i-1}(V)$ is $r+s-i$.
\end{theorem}
\begin{proof}
It follows from Lemma \ref{le:hyperplane} that the rank variety of
$S^{p-1}(V)^{(i)}$ is the hyperplane given by the vanishing of 
the $i$th row of the above matrix. Now
apply Lemma \ref{le:Steinberg}.

Now by the usual Vandermonde argument, given elements
$a_1,\dots,a_s\in k$, the determinant of the matrix
\[ \begin{pmatrix} a_1 & \dots & a_s \\
a_1^p &  & a_s^p \\
\vdots & & \vdots \\
a_1^{p^{s-1}} & & a_s^{p^{s-1}}
\end{pmatrix} \]
is, up to non-zero scalar, the product of the non-zero 
$\bbF_p$-linear combinations of $a_1,\dots,a_s$, 
one from each one dimensional subspace. It therefore vanishes if and
only if they are linearly dependent over $\bbF_p$.

Applying this to the lower right corner of the matrix in the theorem, the 
linear independence of the rows of this matrix follows
using the fact that the $\mu_i$ are linearly
independent over $\bbF_p$. Alternatively, this can be deduced from Theorem~\ref{th:sympowers-low}.
\end{proof}

\begin{proposition}\label{pr:induced}
Let $M$ be a $p$-dimensional uniserial $kH$-module. Then there is a
subalgebra $A$ of $kH$ of dimension $p^{r+s-1}$ with the following
properties:
\begin{enumerate}
\item[\rm (i)] $kH$ is flat as an $A$-module,
\item[\rm (ii)] the restriction of $M$ to $A$ is a direct sum of $p$
  copies of $k$ with trivial action, and
\item[\rm (iii)] $M$ is isomorphic to $kH\otimes_{A} k$ as a $kH$-module.
\end{enumerate}
\end{proposition}
\begin{proof}
Let $I\subseteq kH$ be the annihilator of $M$. Then $I$ is an ideal of
codimension $p$, and $M$ is isomorphic to $kH/I$. Furthermore, for
$n\ge 0$ we have
$\Rad^n(M) = J^n(kH).M$, and so $M/\Rad^n(M) \cong kH/(I+J^n(kH))$.
Since $M/\Rad^2(M)$ has dimension two, so does $kH/(I+J^2(kH))$, and
therefore $(I+J^2(kH))/J^2(kH)$ has dimension $r+s-1$. As a vector
space, this is isomorphic to $I/(I\cap J^2(kH))$. 
Choose elements $u_1,\dots,u_{r+s-1}\in I$ which are linearly
independent modulo $J^2(kH)$, and let 
$A=k[u_1,\dots,u_{r+s-1}]\subseteq kH$. Then $kH$ is flat as an
$A$-module, and $A$ acts trivially on $M$. So we have
$\dim_k\Hom_{A}(k,M)=p$, and therefore
\[ \dim_k\Hom_{kH}(kH\otimes_{A}k, M)=p. \] 
Similarly, we have
\[ \dim_k\Hom_{kH}(kH\otimes_{A}k,\Rad(M))=p-1. \]
There is therefore a homomorphism from $kH\otimes_{A}k$ to $M$ whose
image does not lie in $\Rad(M)$. Both modules are uniserial of length
$p$, so such a homomorphism is necessarily an isomorphism.
\end{proof}

\begin{theorem}\label{th:annihilators}
\begin{enumerate}
\item[\rm (i)]
There exists a flat embedding $A\to kH$ of a
subalgebra $A$ of dimension $p^{r+s-1}$ and an isomorphism 
$S^{p-1}(V^\#) \cong kH \otimes_{A} k$. 
\item[\rm (ii)]
The cohomology $H^*(kH,S^{p-1}(V^\#))$ is annihilated by
\[ -x_1+\mu_1^pz_1+\dots+\mu_r^pz_r. \]
\item[\rm (iii)]
More generally, for $1\le i\le r+s$, there exists a flat embedding
$A_i \to kH$ of a subalgebra $A_i$ of dimension $p^{r+s-i}$ and an
isomorphism $S^{p^i-1}(V^\#) \cong kH \otimes_{A_i} k$. 
The cohomology $H^*(kH,S^{p^i-1}(V^\#))$ is
annihilated by the first $i$ elements of the regular sequence
\begin{align*}
-x_1\qquad\ \qquad\ \qquad\ &+\mu_1^pz_1+\dots+\mu_r^pz_r \\
-x_2\qquad\ \qquad\ &+\mu_1^{p^2}z_1+ \dots + \mu_r^{p^2}z_r \\
\ddots \qquad\ &\qquad\cdots \\
-x_r&+\mu_1^{p^r}z_1 + \dots + \mu_r^{p^r}z_r \\
&\phantom{{}+{}}\mu_1^{p^{r+1}}z_1+\dots + \mu_r^{p^{r+1}}z_r \\
&\qquad\cdots \\
&\phantom{{}+{}}\mu_1^{p^{r+s}}z_1+\dots+\mu_r^{p^{r+s}}z_r.
\end{align*}
\end{enumerate}
\end{theorem}
\begin{proof}
(i) This follows from Lemma \ref{le:hyperplane} and Proposition
\ref{pr:induced}.

(ii) The annihilator of
cohomology consists of the elements of cohomology of $H$ whose
restriction to $A$ is zero, and is therefore generated by a degree one
element and its image under $\beta\cP^0$.
Taking into account the Frobenius twist in the
relationship between rank variety and cohomology variety for an
elementary abelian $p$-group, the
statement follows from Lemma \ref{le:hyperplane}.

(iii) This follows in the same way, using Lemma \ref{le:Steinberg} and
Theorem \ref{th:rank-variety}.
\end{proof}

\section{Proof of the main theorem}\label{se:proof}

In this section, we prove Theorem \ref{th:main}, using the results of
the previous sections. 
\begin{theorem}\label{th:main}
Let $G$ be the semidirect product \[ (\bbG_a^-\times \bbG_a^-) \rtimes H \]
where $H= \bbG_{a(r)} \times (\bbZ/p)^s$ acts faithfully.
Then there is a non-zero element $\zeta\in
H^{1,1}(G,k)$ such that for all $u\in H^{1,0}(G,k)$ we have $\beta\cP^0(u).\zeta^{p^{r+s-1}(p-1)}=0$.
\end{theorem}

\begin{proof}
In contrast with the case $H=\bbG_{a(1)}$ studied in Section
\ref{se:Ga(1)}, for more general $H$ we only have one copy of $\mu_p=\bbG_{m(1)}$ acting
as automorphisms. This acts by scalar multiplication on the generators
$u$ and $v$ of $k(\bbG_a^- \times \bbG_a^-)$ and centralizes $H$. So it
also acts by scalar multiplication on the generators $\zeta$ and
$\eta$ in $H^{1,1}(\bbG_a^-\times \bbG_a^-,k)=k[\zeta,\eta]$. As in
Section \ref{se:Ga(1)} we use weights in $\bbZ/(p-1)$ for this action.
So $\Vert\zeta\Vert=\Vert\eta\Vert=1$, and everything in
$H^{*,*}(H,k)$ has weight zero.

We compare two spectral sequences. The first is the spectral sequence
\begin{equation}\label{eq:ss1}
H^{*,*}(\bbG_a^- \times H, H^{*,*}(\bbG_a^-,k)) \Rightarrow H^{*,*}(G,k),
\end{equation}
associated with the central extension 
\[ 1 \to \bbG_a^- \to G \to \bbG_a^- \times H \to 1\] 
The second is the spectral sequence of the semidirect product
\begin{equation}\label{eq:ss2}
H^{*,*}(H,H^{*,*}(\bbG_a^- \times \bbG_a^-,k))\Rightarrow H^{*,*}(G,k).
\end{equation}
As in Section \ref{se:Ga(1)}, the differentials in these spectral
sequences have to preserve weights for the action of $\mu_p$.

In the first spectral sequence \eqref{eq:ss1}, we have 
\[ d_2(\eta) = (\lambda_1+\mu_1 y_1+\dots + \mu_s y_s)\zeta. \]
Applying the Kudo transgression theorem, we get
\[ d_{p+1}(\eta^p) = \cP^\shalf d_2(\eta) = (\lambda_2 + \mu_1^p y_1 + \dots + \mu_s^p
  y_s)\zeta^p. \]
Continuing this way,
\begin{align*} 
d_2(\eta) &= (\lambda_1+\qquad\qquad\qquad\mu_1 y_1\ +\ \dots\  +\  \mu_s
            y_s)\zeta. \\
d_{p+1}(\eta^p) &= \qquad (\lambda_2 +\qquad\qquad 
\mu_1^p y_1\ +\ \dots\ +\ \mu_s^p y_s)\zeta^p \\
\cdots& \qquad\qquad\qquad\ddots\qquad\qquad\cdots \\
d_{p^{r-1}+1}(\eta^{p^{r-1}}) & =\qquad\qquad\qquad (\lambda_r + \mu_1^{p^{r-1}}\!y_1 +
\dots + \mu_s^{p^{r-1}}\!y_s)\zeta^{p^{r-1}} \\
d_{p^r+1}(\eta^{p^r}) &=\qquad\qquad\qquad\qquad\ (\mu_1^{p^r}y_1\ +
\ \dots\ +\ \mu_s^{p^r}y_s)\zeta^{p^r} \\
\cdots & \qquad\qquad\qquad\qquad\qquad\qquad\cdots \\
d_{p^{r+s-1}+1}(\eta^{p^{r+s-1}}) &=\qquad\qquad\qquad\qquad\ (\mu_1^{p^{r+s-1}}\!y_1+\dots+
\mu_s^{p^{r+s-1}}\!y_s)\zeta^{p^{r+s-1}}
\end{align*}
and finally $d_{p^{r+s}}(\eta^{p^{r+s}})$ is in the ideal generated by
the previous ones, so $\eta^{p^{r+s}}$ is a universal cycle.

Applying Corollary~\ref{co:calculation} to the In the restriction of the first 
spectral sequence \ref{eq:ss1} to the
semidirect product of $\bbG_a^-\times \bbG_a^-$ with a minimal
subgroup of $H$ we conclude that 
\begin{equation}\label{eq:dp} 
d_p((\lambda_1+\mu_1y_1+\dots +\mu_sy_s)\eta^{p-1}) 
\end{equation}
is non-zero. It has to be something of weight $p-1$, and is therefore
something times $\zeta^{p-1}$.

Now, in the second spectral sequence \ref{eq:ss2}, Theorem
\ref{th:annihilators} shows that the element 
\[ -x_1+\mu_1^p z_1+\dots + \mu_s^p z_s \]
on the base annihilates $\zeta^{p-1}$ on the fibre in the $E_2$
page. This means that in $H^{*,*}(G,k)$, this product is zero modulo
smaller powers of $\zeta$. 

Putting these two pieces of information together, we see that
\eqref{eq:dp} has to be a non-zero multiple of 
$(-x_1+\mu_1^p z_1+\dots + \mu_s^p z_s)\zeta^{p-1}$.
Therefore, in $H^{*,*}(G,k)$ we have the relation
\[ (-x_1+\mu_1^p z_1+\dots + \mu_s^p z_s)\zeta^{p-1}=0. \]

We now apply Steenrod operations to this relation to obtain further relations.
Applying $\cP^{\frac{p-1}{2}}$, we obtain
\[  (-x_2+\mu_1^{p^2} z_1+\dots + \mu_s^{p^2}
  z_s)\zeta^{p^2-p}=0. \]

Continuing this way, applying $\cP^{\frac{p(p-1)}{2}},\cP^{\frac{p^2(p-1)}{2}},\dots$ we have

\begin{align*}
(-x_1+\qquad\qquad\qquad\quad\mu_1^p z_1+\dots + \mu_s^p z_s)\zeta^{p-1}\qquad&=0 \\
\qquad (-x_2+\qquad\quad\mu_1^{p^2} z_1+\dots + \mu_s^{p^2} z_s)\zeta^{p(p-1)}\quad&=0 \\
\cdots\qquad\qquad\qquad\cdots& \\
\qquad\qquad\quad(-x_r+\mu_1^{p^r}z_1+\dots+\mu_s^{p^r}z_s)\zeta^{p^{r-1}(p-1)}&=0  \\
\qquad \qquad\qquad\quad\ (\mu_1^{p^{r+1}}\!z_1+\dots+\mu_s^{p^{r+1}}\!z_s)\zeta^{p^r(p-1)}&=0 \\
\qquad\qquad\qquad\qquad\qquad\qquad\cdots& \\
\qquad\qquad\qquad\quad\ (\mu_1^{p^{r+s}}\!z_1+\dots+\mu_s^{p^{r+s}}\!z_s)\zeta^{p^{r+s-1}(p-1)}&=0.
\end{align*}

Every linear combination of $x_1,\dots,x_r,z_1,\dots,z_s$ is spanned
by the coefficients of the powers of $\zeta$. In
particular, this shows that every $x_i\zeta^{p^{r+s-1}(p-1)}$ and every
$z_i\zeta^{p^{r+s-1}(p-1)}$ is zero in $H^{*,*}(G,k)$. This completes the
proof.
\end{proof} 

As a last result of this note, we deduce a corollary to be used in 
\cite{Benson/Iyengar/Krause/Pevtsova:bikp5}. 

\begin{corollary}\label{co:bikp5}
Let $G$ be a finite unipotent supergroup scheme, with a normal sub-super\-group
scheme $N$ such that $G/N \cong \bbG_a^- \times \bbG_{a(r)} \times
(\bbZ/p)^s$. If the inflation map $H^{1,*}(G/N,k)\to H^{1,*}(G,k)$ is
an isomorphism and $H^{2,1}(G/N,k) \to H^{2,1}(G,k)$ is
not injective then there exists a non-zero element $\zeta\in
H^{1,1}(G,k)$ such that for all $u\in H^{1,0}(G,k)$ we have $\beta\cP^0(u).\zeta^{p^{r+s-1}(p-1)}=0$.
\end{corollary}
\begin{remark} The condition  that the inflation map is an isomorphism on $H^{1,*}$ effectively 
decodes the fact that $G/N$ is the maximal quotient of prescribed form.  See 
\cite{Benson/Iyengar/Krause/Pevtsova:bikp5} for more details on how it arises. 
\end{remark}

\begin{proof} Recall that $H^{*,*}(G.N,k) \cong  k[\zeta]\otimes k[x_1,\dots,x_r] \otimes
\Lambda(\lambda_1,\dots,\lambda_r) \otimes
k[z_1,\dots,z_s] \otimes \Lambda(y_1,\dots,y_s)$ with $\zeta$ in degree $(1,1)$ and the 
rest of the generators in even internal degree.  
If $H^{2,1}(G/N,k)\to H^{2,1}(G,k)$ is not an
isomorphism then the kernel contains an element of the form $u\zeta$ with
$u\in H^{1,0}(G/N,k)$, $\zeta\in H^{1,1}(G/N,k)$. The five term sequence corresponding 
to the extension $1 \to N \to G \to G/N \to 1$, 
\[\xymatrix{H^{1,1}(G/N,k)\ar[r] & H^{1,1}(G,k) \ar[r] & H^{1,1}(N,k)^G \ar[r]^-{d_2} &  H^{2,1}(G/N,k) \ar[r] &H^{2,1}(G,k)}\] 
gives an element $0 \ne\eta\in H^{1,1}(N,k)^G$
such that $d_2(\eta)=u\zeta$. Now $H^{1,1}(N,k) \cong \Hom(N,\bbG_a^-)$ 
(see \cite[Lemma 4.1]{Benson/Iyengar/Krause/Pevtsova:bikp5}), so 
corresponding to $\eta$ there is a $G$-invariant surjective homomorphism
$N \to \bbG_a^-$. Letting $N_1\le N$ be the kernel of this
homomorphism, it follows that $N_1$ is normal in $G$. Looking at the
map of five term sequences given by factoring out $N_1$, we see that
we might as well replace $G$ by $G/N_1$ and $N$ by $N/N_1$, since the
hypotheses of the corollary are preserved, and the conclusion for
$G/N_1$ inflates to the same conclusion for $G$.

We are left in a situation where we have a short exact sequence
\[ \xymatrix{1 \ar[r] & N \ar[r] \ar[d]^\cong & G \ar[r] \ar[d]^= &
    G/N \ar[r] \ar[d]^\cong & 1 \\
    1 \ar[r] & \bbG_a^- \ar[r] & G \ar[r] & \bbG_a^- \times
    \bbG_{a(r)} \times (\bbZ/p)^s \ar[r] & 1.} \]
The fact that $d_2(\eta)=u\zeta$ means that the restrictions of
$d_2(\eta)$ to the two factors $\bbG_a^-$ and
$\bbG_{a(r)}\times(\bbZ/p)^s$ of the quotient are both zero.
So the restriction of the extension to these two factors gives abelian
subgroups. It is then easy to see that the restricted extensions
split, and so $G$ has subgroups $\bbG_a^-\times \bbG_a^-$ and
$\bbG_{a(r)} \times (\bbZ/p)^s$ satisfying the conditions for a
semidirect product. This puts us in the situation of Theorem
\ref{th:main}, and the Corollary is proved.
\end{proof}

\bibliographystyle{amsplain}
\bibliography{repcoh}

\end{document}